\documentclass[a4paper]{amsart}

\usepackage{amsmath}
\usepackage{amssymb}
\usepackage[left=3cm,right=3cm,top=3cm,bottom=2cm]{geometry}

\usepackage[all]{xy}
\usepackage{epsfig}

\def\MM{\mathcal{M}}

\newtheorem{thm}{Theorem}[section]

\newtheorem{lem}[thm]{Lemma}

\theoremstyle{definition}
\newtheorem{rem}[thm]{Remark}

\title{The mutation class of $D_n$ quivers}

\author{Dagfinn F. Vatne}

\address{Institutt for matematiske fag, Norges teknisk-naturvitenskapelige
  universitet, N-7491 Trondheim, Norway}

\email{dvatne@math.ntnu.no}

\begin{document}

\maketitle

\begin{abstract}
We give an explicit description of the mutation classes of quivers of type
$D$.
\end{abstract}

\section*{Introduction}

In their very influential work on cluster algebras, Fomin and Zelevinsky
defined the notion of \emph{matrix mutation} \cite{fz1}. When applied to
skew-symmetric matrices, this can be interpreted as an operation on quivers
(i.e.\ directed graphs), and this is called \emph{quiver mutation}. The
quivers whose underlying graphs are simply laced Dynkin diagrams have a special
significance in the cluster algebra theory, as they appear in the finite type
classification \cite{fz2}.

The purpose of this note is to give an explicit description of the
\emph{mutation class} of $D_n$ quivers, for $n\geq 4$. That is, we will
present the set of quivers which can be obtained by iterated mutation on a
quiver whose underlying graph is of Dynkin type $D$. It turns out that the
quivers in these mutation classes are easily recognisable. In particular, our
result gives a complete description of the cluster-tilted algebras of type
$D$, see \cite{bmr1,bmr2,ccs2}.

The shape of these quivers can be deduced from Schiffler's geometric model for
the cluster categories of type $D$ \cite{sch}. Nevertheless, it is convenient
to have an explicit description. The method used in this paper to obtain the
description is purely combinatorial, and no prerequisites are needed.

In order to understand these mutation classes, it is necessary to understand the
mutation classes of quivers of Dynkin type $A$. These are explicitly
described in \cite{seven} and \cite{bv}, and are also implicit in
\cite{ccs}. They will be recalled in Section \ref{sec:mutationclass}.

\section{Quivers and mutation} \label{sec:quivermutation}

In this section we will briefly recall the definition of quiver mutation from
\cite{fz1}. Before that, we fix some standard terminology about quivers which
will be used here.

A \emph{quiver} $Q$ is a directed graph, that is, a quadruple
$(Q_0,Q_1,h,t)$ consisting of two sets $Q_0$ and $Q_1$, whose elements
are called \emph{vertices} and \emph{arrows} respectively, and two functions
$h,t:Q_1\to Q_0$ (assigning a ``head'' and a ``tail'' to each arrow). We
will often think of $Q$ as the union of $Q_0$ and $Q_1$, and keep in mind
that each element of $Q_1$ connects two vertices and has a direction. If
either $h(\alpha)=i$ or $t(\alpha)=i$, we say that $\alpha$ is \emph{incident}
with $i$. Moreover, if $t(\alpha)=j$ and $h(\alpha)=i$, we will say that
$\alpha$ is an arrow \emph{from $j$ to $i$}.

For any vertex $i$ in $Q_0$, the \emph{valency} of $i$ (in $Q$) is the number
of neighbouring vertices, i.e.\ the number of vertices $j\not =i$ such that
there exists an arrow $\alpha \in Q_1$ with either $h(\alpha)=i$ and
$t(\alpha)=j$ or vice versa.

We will assume throughout that quivers do not have loops or 2-cycles. In other
words, for any $\alpha \in Q_1$, we have that $h(\alpha) \not = t(\alpha)$ and
there does not exist $\beta \in Q_1$ such that $h(\beta)=t(\alpha)$ and
$t(\beta)=h(\alpha)$.

For a quiver $Q=(Q_0,Q_1,h,t)$ and a vertex $i\in Q_0$, the quiver
$\mu_i(Q)=(Q_0^*,Q_1^*,h^*,t^*)$ is obtained by making the following changes
to $Q$:
\begin{itemize}
\item All arrows incident with $i$ are reversed, i.e.\ $h^*(\alpha)=t(\alpha)$
  and $t^*(\alpha)=h(\alpha)$ for such arrows.
\item Whenever $j,k\in Q_0$ are such that there are $r>0$ arrows from $j$ to
  $i$ (in $Q$) and $s>0$ arrows from $i$ to $k$ (in $Q$), first add $rs$ arrows
  from $j$ to $k$. Then remove a maximal number of 2-cycles.
\end{itemize}
There is a choice involved in the last step in this procedure, but this will
not concern us, as the possible resulting quivers are isomorphic as quivers,
and will be regarded as equal. The quiver $\mu_i(Q)$ so defined is said to be
the quiver obtained from $Q$ by \emph{mutation} at $i$.

Mutation at a vertex is an involution, that is, $\mu_i(\mu_i(Q))=Q$. It follows
that mutation generates an equivalence relation on quivers. Two quivers are
\emph{mutation equivalent} if one can be obtained from the other by some
sequence of mutations. An equivalence class will be called a \emph{mutation
  class}.

The following lemma is well known and easily verified:

\begin{lem}
If the quivers $Q_1$ and $Q_2$ both have the same underlying graph $T$, and
$T$ is a tree, then $Q_1$ and $Q_2$ are mutation equivalent.
\end{lem}

In particular, by this lemma, it makes sense to speak of the mutation classes
of the simply laced Dynkin diagrams. In this paper we will be concerned with
the mutation classes of type $A_k$ quivers:
\begin{equation} \label{eq:typeAquiver}
\xymatrix{
1\ar[r]&2\ar[r]&3\ar[r]&\cdots\ar[r]&(k-1)\ar[r]&k
}
\end{equation}
and with the mutation classes of type $D_n$ quivers:
\begin{equation} \label{eq:typeDquiver}
\xymatrix{
&&&&&(n-1)\\
1\ar[r]&2\ar[r]&3\ar[r]&\cdots\ar[r]&(n-2)\ar[ur]\ar[dr]&\\
&&&&&n
}
\end{equation}
By the cluster algebra finite-type classification \cite{fz2}, we know that
these mutation classes are finite.

\section{The set $\MM^D_n$} \label{sec:mutationclass}

Let first $\MM^A_k$ be the mutation class of $A_k$. This set of quivers
consists of the connected quivers that satisfy the following \cite{bv}:
\begin{itemize}
\item there are $k$ vertices,
\item all non-trivial cycles are oriented and of length 3,
\item a vertex has valency at most four,
\item if a vertex has valency four, then two of its adjacent arrows belong
  to one 3-cycle, and the other two belong to another 3-cycle,
\item if a vertex has valency three, then two of its adjacent arrows belong to
  a 3-cycle, and the third arrow does not belong to any 3-cycle
\end{itemize}
By a \emph{cycle} in the first condition we mean a cycle in the underlying
graph, not passing through the same edge twice. The union of all $\MM^A_k$ for
all $k$ will be denoted by $\MM^A$. Note that for a quiver $\Gamma$ in
$\MM^A$, any connected subquiver of $\Gamma$ is also in $\MM^A$.

For a quiver $\Gamma$ in $\MM^A$, we will say that a vertex $v$ is a
\emph{connecting} vertex if $v$ has valency at most 2 and, moreover, if $v$ has
valency 2, then $v$ is a vertex in a 3-cycle in $\Gamma$.

We now define a class $\MM^D_n$ of quivers which will be shown to be the
mutation class of $D_n$. We define $\MM^D_n$ to be the set of quivers $Q$ with
$n$ vertices belonging to one of the following four types:

\textbf{Type I:} $Q$ has two vertices $a$ and $b$ which have valency
one and both $a$ and $b$ have an arrow to or from the same vertex $c$, and
$Q'=Q\backslash \{ a,b\}$ is in $\MM^A_{n-2}$ and $c$ is a connecting vertex
for $Q'$.
\begin{center}
\includegraphics[height=3cm]{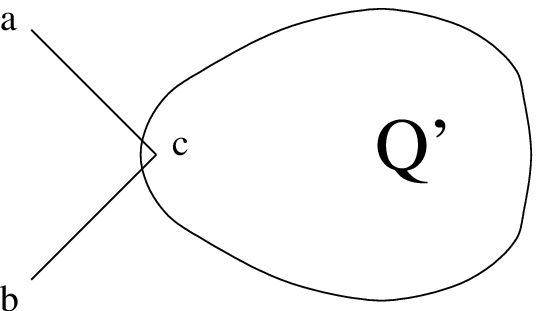}
\end{center}

\textbf{Type II:} $Q$ has a full subquiver $Q_1$ with four vertices which
looks like
\begin{center}
\includegraphics[height=2cm]{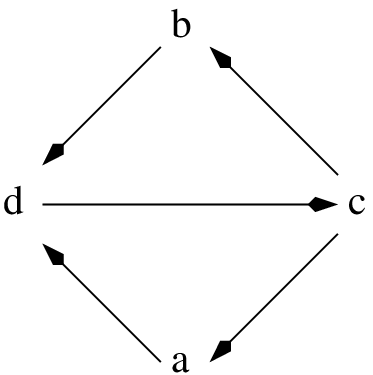}
\end{center}
where the vertices $a$ and $b$ have valency 2 in $Q$, and $Q\backslash
\{a,b,d\to c\}=Q' \overset{\cdot}{\cup} Q''$ is disconnected with two
components $Q'$ and $Q''$ which are both in $\MM^A$ and for which $c$ and $d$
are connecting vertices.
\begin{center}
\includegraphics[height=3cm]{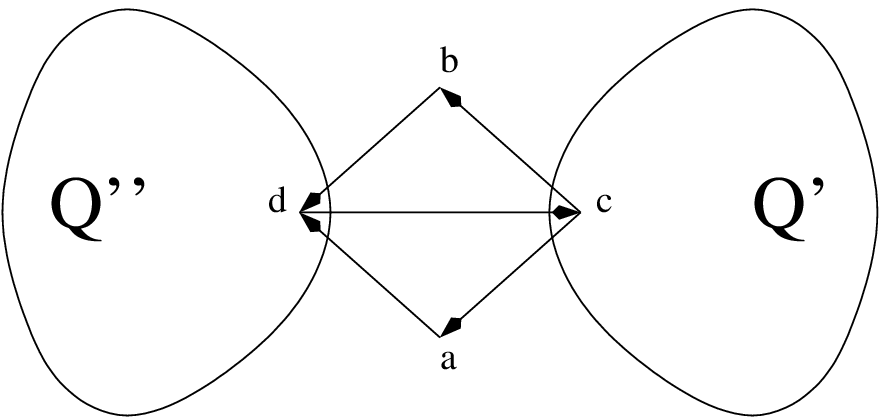}
\end{center}

\textbf{Type III:} $Q$ has a full subquiver which is a directed 4-cycle:
\begin{center}
\includegraphics[height=2cm]{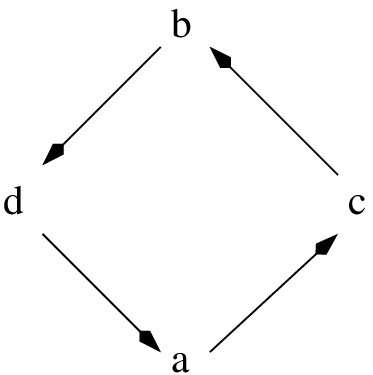}
\end{center}
and $Q\backslash \{a,b\}=Q' \overset{\cdot}{\cup} Q''$ is disconnected with two
components $Q'$ and $Q''$ which are both in $\MM^A$ and for which $c$ and $d$
are connecting vertices, like for Type II.
\begin{center}
\includegraphics[height=3cm]{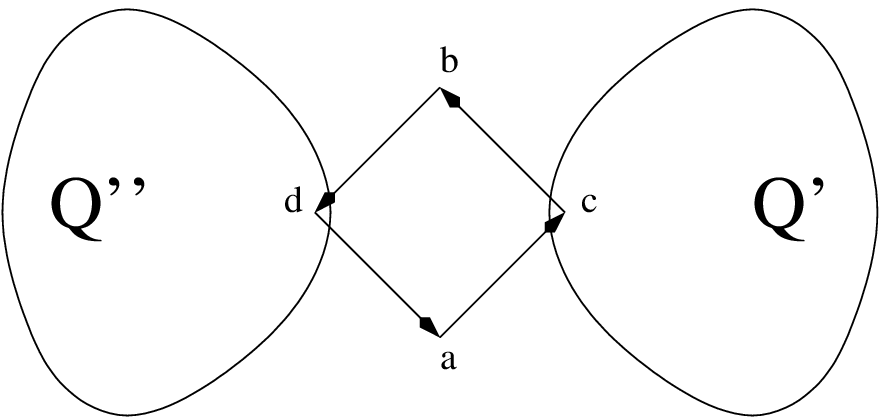}
\end{center}

\textbf{Type IV:} $Q$ has a full subquiver which is a directed $k$-cycle,
where $k\geq 3$. We will call this the \emph{central} cycle. For each arrow
$\alpha :a\to b$ in the central cycle, there may (and may not) be a vertex
$c_{\alpha}$ which is not on the central cycle, such that there is an oriented
3-cycle traversing $a\overset{\alpha}{\to} b\to c_{\alpha}\to a$. Moreover,
this 3-cycle is a full subquiver. Such a 3-cycle will be called a
\emph{spike}. There are no more arrows starting or ending in vertices on the
central cycle.

Now $Q\backslash \{\textrm{vertices in the central
  cycle and their incident arrows}\}=Q'\overset{\cdot}{\cup}
Q''\overset{\cdot}{\cup} Q'''\overset{\cdot}{\cup}\cdots$ is a disconnected
union of quivers, one for each spike, which are all in $\MM^A$, and for which
the corresponding vertex $c$ is a connecting vertex:
\begin{center}
\includegraphics[height=6.5cm]{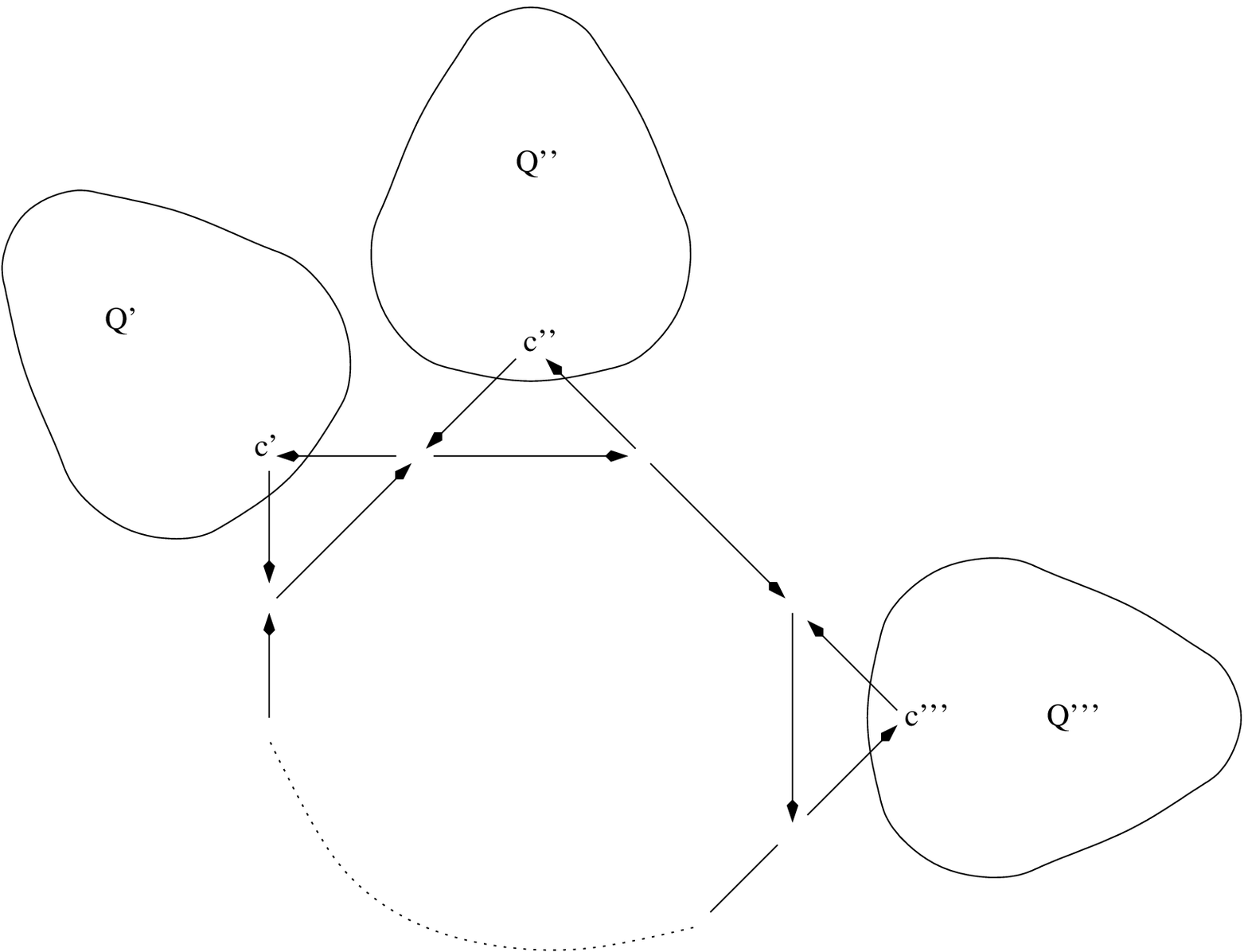}
\end{center}

\begin{rem}
One should note that in Types II, III and IV, the subquivers $Q', Q'',...$ can
be in the set $\MM^A_1$, i.e. they can have only one vertex.
\end{rem}

\begin{rem}
Ringel \cite{r} has described all self-injective cluster tilted
algebras, and in particular shown that they are of type $D$. Their Gabriel
quivers are all Type IV in our description; they are the special cases with
either no spikes, or with a maximal number of spikes and with the subquivers
$Q',Q'',...$ with only one vertex.
\end{rem}

\section{Proof that $\MM^D_n$ is the mutation class of $D_n$}

The main result of this paper is the following:

\begin{thm} \label{thm:mutationclass}
For any $n\geq 4$, a quiver $Q$ is mutation equivalent to $D_n$ if and only if
it is in $\MM^D_n$.
\end{thm}

We will break the proof of Theorem \ref{thm:mutationclass} into small
lemmas. The first takes care of the type $A$ pieces.

\begin{lem} \label{lem:typeApiece}
Let $\Gamma \in \MM^A_k$ for some $k\geq 2$, and let $c$ be a connecting
vertex for $\Gamma$. Then there exists a sequence of mutations on $\Gamma$
satisfying the following:
\begin{itemize}
\item $\mu_c$ does not appear in the sequence
\item the resulting quiver is isomorphic with the quiver in
  \eqref{eq:typeAquiver}
\item under this isomorphism, $c$ is relabelled as 1
\end{itemize}
\end{lem}

\begin{proof}
We will prove the claim by induction on $k$. It is readily checked for small
values.

Assume first that $c$ has valency 1. Noting that the neighbour $c'$ of $c$ is
a connecting vertex for the quiver $\Gamma' = \Gamma \backslash \{ c\}$, we
have by induction that $\Gamma'$ can be mutated to a linearly oriented $A_{k-1}$
quiver without mutating at $c'$. The result looks like this after the
relabelling:
\[
\xymatrix{
c\ar@{-}[r]&1\ar[r]&2\ar[r]&\cdots\ar[r]&(k-2)\ar[r]&(k-1)
}
\]
If the edge between $c$ and 1 is correctly oriented, we are done. If not,
perform the sequence of mutations $\mu_1,\mu_2,...,\mu_{(k-2)},\mu_{(k-1)}$
and relabel to get the quiver in \eqref{eq:typeAquiver}.

Assume now the other possibility, namely that $c$ has valency 2. It is then
traversed by a 3-cycle. Again we note that the neighbouring vertices $c'$ and
$c''$ are connecting vertices for their respective components of the quiver
$\Gamma' = \Gamma \backslash \{ c, c'\to c''\}$, we can apply the induction
hypothesis and perform a sequence of mutations to produce the following
quiver:
\[
\xymatrix{
&c'\ar[dl]\ar[r]&1'\ar[r]&2'\ar[r]&\cdots\ar[r]&k'_1 \\
c\ar[dr]&&&&& \\
&c''\ar[uu]\ar[r]&1''\ar[r]&2''\ar[r]&\cdots\ar[r]&k''_2
}
\]
Now the sequence $\mu_{c'},\mu_{1'},\mu_{2'},...,\mu_{k'_1}$ and a
relabelling will yield the desired result.
\end{proof}

\begin{lem} \label{lem:typeI}
All quivers of Type I are mutation equivalent to $D_n$ quivers.
\end{lem}

\begin{proof}
If we have a quiver $Q$ as in the description of Type I, then $c$ is
connecting for $Q'$. An application Lemma \ref{lem:typeApiece} shows that we
can mutate $Q'$ to a quiver with underlying graph $A_{n-2}$ by mutating in
vertices not equal to $c$. This induces a mutation of $Q$ into a quiver
with underlying graph $D_n$.
\end{proof}

\begin{lem} \label{lem:typeII}
All quivers of Type II are mutation equivalent to $D_n$ quivers.
\end{lem}

\begin{proof}
Let $Q$ be a quiver as in the description of Type II. The vertices $c$ and $d$
are connecting for $Q'$ and $Q''$, so we can apply Lemma \ref{lem:typeApiece}
to each of these. Thus we can mutate $Q$ into the following quiver, for some 
$0\leq m\leq n-4$:
\[
\xymatrix{
&&&&b\ar[dl]&&&& \\
1&\cdots\ar[l]&m\ar[l] &d\ar[l]\ar[rr]&&c\ar[r]\ar[ul]\ar[dl]
&(m+1)\ar[r]&\cdots\ar[r]&(n-4) \\
&&&&a\ar[ul]&&&& 
}
\]
The sequence $\mu_d,\mu_m,\mu_{(m-1)},...,\mu_1$ of mutations will now result
in a quiver with underlying graph $D_n$.
\end{proof}

\begin{lem} \label{lem:typeIII}
All quivers of Type III are mutation equivalent to $D_n$ quivers.
\end{lem}

\begin{proof}
Let $Q$ be any quiver $Q$ as in the description of Type III. If we mutate at
the vertex $a$, we get a quiver $\mu_a(Q)$ of Type II. By Lemma
\ref{lem:typeII}, we have that $\mu_a(Q)$ is mutation equivalent to $D_n$, and
therefore so is $Q$.
\end{proof}

\begin{lem} \label{lem:cycle}
The oriented cycle of length $n$ is mutation equivalent to $D_n$.
\end{lem}

\begin{proof}
Starting with the quiver \eqref{eq:typeDquiver}, we get an oriented cycle by
performing the following sequence of mutations:
$\mu_{(n-1)},\mu_{(n-2)},...,\mu_1$.
\end{proof}

\begin{lem}
All quivers of Type IV are mutation equivalent to $D_n$ quivers.
\end{lem}

\begin{proof}
Let $Q$ be a quiver as in the description of Type IV. By Lemma \ref{lem:cycle}
it is sufficient to show that $Q$ is mutation equivalent to an oriented cycle.

Let $a\to b\to c\to a$ be a spike in $Q$ with $a\to b$ on the central cycle,
and let $Q_c$ be the corresponding type $A$ piece. Since $c$ is connecting for
$Q_c$, we can mutate $Q_c$ (without mutating at $c$) to a linearly oriented
$A_k$ quiver, for some $k$. This induces an iterated mutation of $Q$ resulting
in a quiver with a full subquiver looking like this:
\[
\xymatrix{
a\ar[dd]&&&&& \\
&c\ar[ul]\ar[r]&1\ar[r]&2\ar[r]&\cdots\ar[r]&k \\
b\ar[ur]&&&&& 
}
\]
Performing the mutations $\mu_c,\mu_1,\mu_2,...,\mu_k$ on this quiver yields a
quiver which is just a directed path from $a$ to $b$. This induces an iterated
mutation on $Q$ which in effect replaces the spike involving $c$ and the
corresponding type $A$ subquiver $Q_c$ with a directed path from $a$ to $b$.

Doing this to all the spikes and the type $A$ pieces, we get an oriented
$n$-cycle, which is what we wanted.
\end{proof}

We now put the pieces together to prove the main result:

\begin{proof}[Proof of Theorem \ref{thm:mutationclass}.]
By the lemmas above, all quivers in $\MM^D_n$ are mutation equivalent to
$D_n$. It remains to show that $\MM^D_n$ is closed under mutation.

Let $Q$ be some quiver in $\MM^D_n$ and let $v$ be some vertex. If $v$ is
inside one of the type $A$ pieces $Q',Q'',...$ as in the descriptions
in Section \ref{sec:mutationclass}, but not the connecting vertex connecting
the piece to the rest of the quiver $Q$, then the mutated quiver $\mu_v(Q)$ is
still of the same type. It therefore suffices to check what happens when we
mutate at the other vertices, and we will consider four cases, according to
the type of $Q$.

\textbf{Case I.} Suppose that $Q$ is of Type I. If $v$ is one of the vertices
$a,b$ (as in the description), then $\mu_a(Q)$ is obviously also of Type
I, as $\mu_v$ only reverses the incident arrow. Assume therefore that $v$ is
the vertex $c$, connecting for $Q'$. If $v$
has valency 2 in $Q'$, then the quiver $\mu_v(Q)$ is of Type II or Type IV
(with a central cycle of length 3), depending on the orientation of the arrows
between $v$ and $a,b$. If $v$ has valency 1 in $Q'$, then the quiver
$\mu_v(Q)$ is either of Type I (if $v$ is a source or a sink in $Q$), Type II
or Type IV (with a central cycle of length 3), depending on the orientations.

\textbf{Case II.} Let now $Q$ be of Type II. For $v$ equal to $a$ or $b$ in the
description of the type, we have that $\mu_v(Q)$ is of Type III. If $v$ is
$c$ and the valency of $v$ in $Q'$ is 2, then then $\mu_v(Q)$ is also of Type
II. If the valency is 1, then $\mu_v(Q)$ is either of Type I or Type II,
depending on the orientation of the arrow incident with $v$ in $Q'$. See
Figure \ref{fig:mut2}. The case when $v$ is $d$ is similar.
\begin{figure}
\centering
\includegraphics[height=3.2cm]{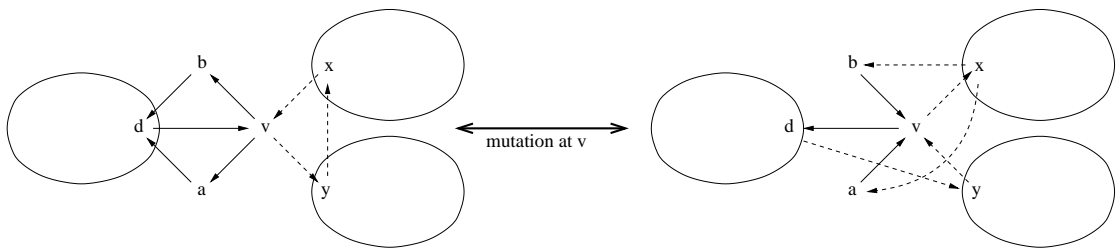}
\caption{Mutation of a Type II quiver. The blobs represent type $A$
  subquivers, and the dotted arrows indicate that they may not appear. If the
  $x$-blob is not there, $\mu_v(Q)$ is Type I. Otherwise, it is Type II.}
  \label{fig:mut2}
\end{figure}

\textbf{Case III.} Suppose $Q$ is of Type III. If $v$ is one of the vertices
$a,b$ in the description, then $\mu_v(Q)$ is of Type II. So assume $v$ is
$c$. (The case of $v$ equal to $d$ is the same.) Then $\mu_v(Q)$ is of Type IV
with a central cycle of length 3. See Figure \ref{fig:mut1}.
\begin{figure}
\centering
\includegraphics[height=3.2cm]{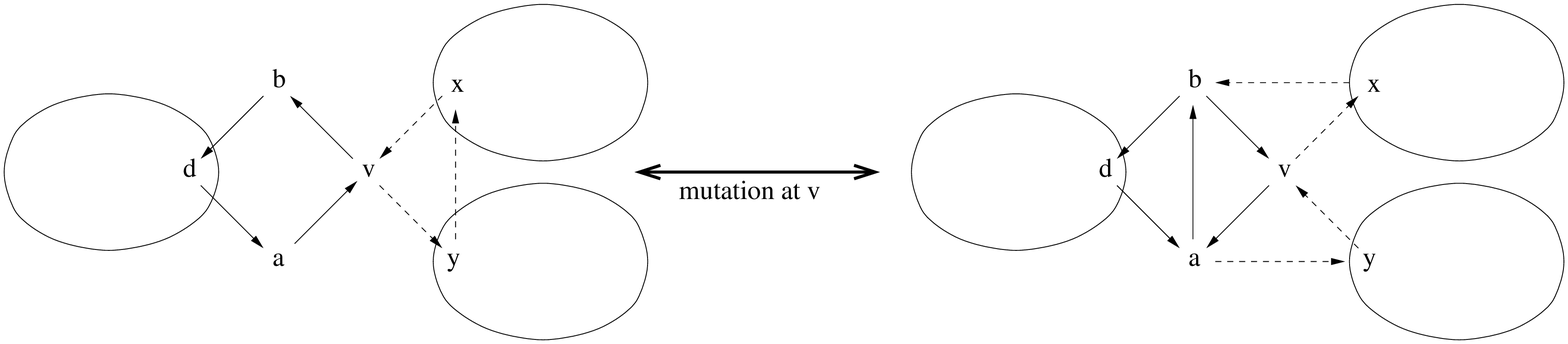}
\caption{If we mutate at $v$, the Type III quiver $Q$ on the left is mutated
  to the Type IV quiver $\mu_v(Q)$ on the right, where the central cycle has
  length 3.} \label{fig:mut1}
\end{figure}

\textbf{Case IV.} Finally, let $Q$ be a quiver of Type IV. First consider a
vertex $v$ which is on the central cycle. If the length of the central cycle
is 3, then $\mu_v(Q)$ is of Type III if the arrow on the opposite side of the
central cycle is part of a spike and Type I if not. See Figure
\ref{fig:mut3}. (Recall our assumtion that $n\geq 4$.) If the length of the
central cycle is 4 or more, then $\mu_v(Q)$ is also of Type IV, and has a
central cycle one arrow shorter than for $Q$.
\begin{figure}
\includegraphics[height=3cm]{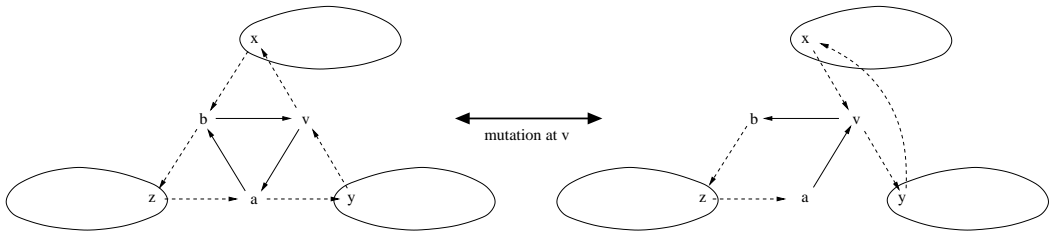}
\caption{If $Q$ (on the left) is Type IV with a central cycle of length 3,
  then $\mu_v(Q)$ is Type I or III, depending on whether there is an opposite
  spike.} \label{fig:mut3}
\end{figure}

If $v$ is a vertex on a spike, but not on the central cycle, then $\mu_v(Q)$
is also Type IV, with a central cycle one arrow longer than for $Q$.

We have now seen that for any quiver $Q\in \MM^D_n$ and any vertex $v$, the
mutated quiver $\mu_v(Q)$ is also in $\MM^D_n$, so the proof is finished.
\end{proof}

\section*{Acknowlegdgement}
The author would like to thank his supervisor Aslak Bakke Buan for many
interesting discussions.

\end{document}